\documentclass[10pt]{amsart}

\usepackage[colorlinks]{hyperref}
\usepackage{color,graphicx,shortvrb}
\usepackage[latin 1]{inputenc}

\usepackage[active]{srcltx} %SRC Specials for DVI search

\usepackage{enumerate}

\newtheorem{theorem}{Theorem}[section]

\newtheorem{corollary}[theorem]{Corollary}

\newtheorem{lemma}[theorem]{Lemma}

\newtheorem{problem}[theorem]{Problem}
\newtheorem{proposition}[theorem]{Proposition}
\newtheorem{remark}[theorem]{Remark}

\def\J#1#2#3{ \left\{ #1,#2,#3 \right\} }

\def\11{\textbf{$1$}}

\DeclareGraphicsExtensions{.jpg,.pdf,.png,.eps}

\usepackage{enumerate}

\begin{document}

\title[Local triple derivations ]{Local triple derivations on C$^*$-algebras and JB$^*$-triples}
\date{}

\author[Burgos]{Maria Burgos}
\email{maria.burgos@uca.es}
\address{Departamento de Matematicas, Facultad de Ciencias Sociales y de la Educacion,  Universidad de Cádiz, 11405, Jerez de la Frontera, Spain.}

\author[Fern\'{a}ndez-Polo]{Francisco J. Fern\'{a}ndez-Polo}
\email{pacopolo@ugr.es}
\address{Departamento de An{\'a}lisis Matem{\'a}tico, Facultad de
Ciencias, Universidad de Granada, 18071 Granada, Spain.}

\author[A.M. Peralta]{Antonio M. Peralta}
\email{aperalta@ugr.es}
\address{Departamento de An{\'a}lisis Matem{\'a}tico, Facultad de
Ciencias, Universidad de Granada, 18071 Granada, Spain.}

\thanks{Authors partially supported by the Spanish Ministry of Economy and Competitiveness,
D.G.I. project no. MTM2011-23843, and Junta de Andaluc\'{\i}a grants FQM0199 and
FQM3737.}



\date{February, 2013}

\maketitle
 \thispagestyle{empty}

\begin{abstract} In a first result we prove that every continuous local triple derivation on a JB$^*$-triple is a triple derivation. We also give an automatic continuity result, that is, we show that local triple derivations on a JB$^*$-triple are continuous even if not assumed a priori to be so. In particular every local triple derivation on a C$^*$-algebra is a triple derivation. We also explore the connections between (bounded local) triple derivations and generalised (Jordan) derivations on a C$^*$-algebra.
\end{abstract}

\section{Introduction}

An (associative) \emph{derivation} from a C$^*$-algebra $A$ into a Banach $A$-bimodule $X$ is a linear mapping $D: A \to X$ satisfying $D(a b) = D(a) b + a D(b),$ for every $a,b$ in $A$. A linear mapping $T : A\to X$ is a \emph{local derivation} if for each $a$ in $A$ there is a derivation $D_a$ from $A$ into $X$ with $D_a (a) = T(a).$ Local derivations were introduced by R.V. Kadison \cite{Kad90} who proved that every continuous local derivation from a von Neumann algebra $M$ (i.e. a C$^*$-algebra which is also a dual Banach space) into a dual Banach $M$-bimodule is a derivation. Kadison's theorem motivated a flourishing line of research which culminates in 2001 with a definite contribution by B.E. Johnson \cite{John01}, who shows that every bounded local derivation from a C$^*$-algebra $A$ into a Banach $A$-bimodule is a derivation. In the just quoted paper, Johnson showed that the continuity hypothesis is, in fact, superfluous by proving that every local derivation from a C$^*$-algebra $A$ into a Banach $A$-bimodule is continuous.\smallskip

Besides of holomorphic motivations coming from the classification of bounded symmetric domains in arbitrary complex Banach spaces (cf. \cite{Ka}), there are additional reasons, from the point of view of Functional Analysis, to study a strictly wider class of complex Banach spaces which comprises the categories of C$^*$-algebras, JB$^*$-algebras and ternary rings of operators, and is known as the category of JB$^*$-triples. For example, contrary to what happens for C$^*$-algebras, the category of JB$^*$-triples is stable under contractive projections (cf. \cite{Sta}, \cite{Ka2} and \cite{FriRu4}). A JB$^*$-triple is a complex Banach space $E$ equipped with a triple product $\J ... : E\times E\times E \to E$ which is conjugate linear in the middle variable and symmetric and linear in the outer variables satisfying certain algebraic-analytic axioms (see Subsection 1.1 for more details). Every C$^*$-algebra is a JB$^*$-triple when equipped with the triple product given by $$ \J xyz =\frac12 (x y^* z +z y^* x).$$

A \emph{triple derivation} on a JB$^*$-triple $E$ is a linear mapping $\delta: E\to E$ satisfying that $$\delta \J abc = \J {\delta(a)}bc + \J a{\delta(b)}c + \J ab{\delta(c)},$$ for every $a,b,c\in E$. A \emph{local triple derivation} on $E$ is a linear map $T : E\to E$ such that for each $a$ in $E$ there exists a triple derivation $\delta_{a}$ on $E$ satisfying $T(a) = \delta_a (a).$ Inspired by the results proved by R.V. Kadison and B.E. Johnson it is natural to consider the following problems:

\begin{problem}\label{problem local triple derivations}
Is every bounded local triple derivation on a JB$^*$-triple $E$ a triple derivation?
\end{problem}

\begin{problem}\label{problem non bounded local triple derivations}
Is every local triple derivation on a JB$^*$-triple $E$ a triple derivation?
\end{problem}

\begin{problem}\label{problem local triple derivations continuous}
Is every local triple derivation on a JB$^*$-triple $E$ continuous?
\end{problem}

We note that the above problems also make sense in the setting of C$^*$-algebras and are interesting questions in its own right.

\begin{problem}\label{problem local triple derivations on Cstar}
Is every (bounded) local triple derivation on a C$^*$-algebra $B$ a triple derivation?
\end{problem}

M. Mackey gave in \cite[Theorem 5.11]{Mack} a positive answer to the above Problem \eqref{problem local triple derivations} under the additional hypothesis of $E$ being a JBW$^*$-triple (i.e. a JB$^*$-triple which is also a dual Banach space). Problems \eqref{problem local triple derivations}, \eqref{problem non bounded local triple derivations} and \eqref{problem local triple derivations continuous} already appear in \cite{Mack} and in \cite{BurFerGarPe2012}. Mackey's theorem can be considered an appropriate version of Kadison's theorem in the setting of JB$^*$-triples. Problem \ref{problem local triple derivations on Cstar} was solved by J. Garcés and the authors of this note for bounded local triple derivations on unital C$^*$-algebras in \cite{BurFerGarPe2012}. The above problems have remained open in their full generality until now.\smallskip

In this paper we completely solve Problems \eqref{problem local triple derivations} and \eqref{problem local triple derivations continuous} (and consequently Problems \eqref{problem non bounded local triple derivations} and \eqref{problem local triple derivations on Cstar}). In Section 2 we prove that every continuous local triple derivation on a JB$^*$-triple is a triple derivation (see Theorem \ref{thm local triple derivations are triple derivations}). Our strategy consists in studying the behavior of the bitranspose, $T^{**}$, on compact and range tripotents in $E^{**}$. It follows from our first main result that Problems \eqref{problem local triple derivations continuous} and \eqref{problem non bounded local triple derivations} are equivalent. In a subsequent result, we explore an automatic continuity result for local triple derivations. T. Barton and Y. Friedman proved in \cite{BarFri} that every triple derivation on a JB$^*$-triple is automatically continuous, we shall adapt their arguments to show that the same statement remains valid for local triple derivations (see Theorem \ref{t automatic continuity}). In Section 3, we particularize our results to the setting of C$^*$-algebras and explore the connections between (local) triple derivations and generalised derivations on a C$^*$-algebra. Among the consequences, we establish appropriate generalizations, to general C$^*$-algebras, of results due to V. Shu'lman \cite{Shu} and J. Li and Z. Pan \cite{LiPan} in the context of unital C$^*$-algebras.

\subsection{Preliminaries}

A JB$^*$-triple is a complex Banach space $E$ together with a
continuous triple product $\J ... : E\times E\times E \to E,$
which is conjugate linear in the middle variable and symmetric
bilinear in the outer variables satisfying the following axioms:
\begin{enumerate}[{\rm $(a)$}] \item (Jordan Identity) \begin{equation}\label{Jordan identity} \J ab{\J xyz} = \J {\J abx}yz - \J x{\J bay}z + \J xy{\J abz},
\end{equation} for all $a,b,x,y,z$ in $E$;
\item If $L(a,b)$ denotes the operator on $E$ given by $L(a,b) x = \J abx,$ the mapping $L(a,a)$ is an hermitian operator with non-negative
spectrum; \item $\|\J aaa\| = \|a\|^3$, for every $a\in E$.\end{enumerate} Given $a,b\in E$, the symbol $Q(a,b)$ will denote the conjugate linear operator defined by $Q(a,b) (x) = \J axb$. We shall write $Q(a)$ instead of $Q(a,a)$.\smallskip

Every C$^*$-algebra is a JB$^*$-triple via the triple product given by
\begin{equation}\label{eq C*-triple product} 2 \J xyz = x y^* z +z y^* x,
\end{equation} and every JB$^*$-algebra is a
JB$^*$-triple under the triple product \begin{equation}\label{eq JB*-alg-triple product} \J xyz = (x\circ y^*) \circ
z + (z\circ y^*)\circ x - (x\circ z)\circ y^*.
\end{equation}

Every element $e$ in a JB$^*$-triple $E$ satisfying $\J eee =e$ is called tripotent.
When a C$^*$-algebra, $A$, is regarded as a JB$^*$-triple, the set of tripotents of $A$ is precisely the set of all
partial isometries in $A$.\smallskip

Associated with each tripotent $e$ in a JB$^*$-triple $E$, there is a
decomposition of $E$ (called \emph{Peirce decomposition}) in the
form: $$E=E_0(e)\oplus E_1(e)\oplus E_2(e),$$ where
$E_k(e)=\{x\in E:L(e,e)x=\frac k2 x\}$ for $k=0,1,2$.
The \emph{Peirce rules} are that $$\J {E_{i}(e)}{E_{j} (e)}{E_{k} (e)}\subseteq E_{i-j+k} (e)$$ if $i-j+k \in \{ 0,1,2\},$ and $\J {E_{i}(e)}{E_{j} (e)}{E_{k} (e)}=\{0\}$ otherwise. Moreover, $\J {E_{2} (e)}{E_{0}(e)}{E} = \J {E_{0} (e)}{E_{2}(e)}{E} =\{0\}.$\smallskip

The Peirce space $E_2 (e)$ is a unital JB$^*$-algebra with unit $e$,
product $x\circ_e y := \J xey$ and involution $x^{*_e} := \J
exe$, respectively. The corresponding \emph{Peirce projections}, $P_{i} (e) : E\to
E_{i} (e)$, $(i=0,1,2)$ are given by $$P_2 (e) = Q(e)^2, \ P_1 (e) = 2 L(e,e) - 2 Q (e)^2, \ \hbox{ and } P_0 (e) = Id-2 L(e,e) + Q (e)^2,$$ where $Id$ is the
identity map on $E$.\smallskip

A JBW$^*$-triple is a JB$^*$-triple which is also a dual Banach space
(with a unique isometric predual  \cite{BarTi}). Since the triple product of every JBW$^*$-triple
is separately weak$^*$ continuous (cf. \cite{BarTi}), it can be easily checked that Peirce projections associated with a tripotent $e$ in a JBW$^*$-triple are weak$^*$ continuous. The second dual, $E^{**}$, of a JB$^*$-triple $E$ is a JBW$^*$-triple \cite{Di86b}.\smallskip

Following standard notation, for each element $a$ in a JB$^*$-triple $E$ we denote $a^{[1]} =
a$ and $a^{[2 n +1]} := \J a{a^{[2n-1]}}a$ $(\forall n\in \mathbb{N})$. It follows from Jordan identity that JB$^*$-triples
are power associative, that is, $\J{a^{[k]}}{a^{[l]}}{a^{[m]}}=a^{[k+l+m]}$. The JB$^*$-subtriple of $E$ generated by the element $a$ will be denoted by the symbol $E_a,$ and coincides with the closed linear span of the elements $a^{[2n-1]}$ with $n\in \mathbb{N}$. The \emph{local Gelfand theory}\label{local Gelfand theory}  for JB$^*$-triples assures that $E_a$ is JB$^*$-triple isomorphic (and hence isometric) to $C_0 (L)$ for some locally compact Hausdorff space $L\subseteq (0,\|a\|],$ such that $L\cup \{0\}$ is compact and $\|a\| \in L$. It is further known that there exists a triple isomorphism $\Psi$ from $E_a$ onto $C_{0}(L),$ satisfying $\Psi (a) (t) = t$ $(t\in L)$ (compare \cite[Lemma 1.14]{Ka}). This result provides us with an spectral resolution for every element in a JB$^*$-triple $E$.\smallskip

The local Gelfand theory for JB$^*$-triples is a powerful tool; among its consequences, it follows that for an element $a$ in a JB$^*$-triple $E$ and
each natural $n$, there exists (a unique) element $a^{[{1}/({2n-1})]}$ in
$E_a$ satisfying $(a^{[{1}/({2n-1})]})^{[2n-1]} = a.$ When $a$ is a
norm one element, the sequence $(a^{[{1}/({2n-1})]})$ converges in
the weak$^*$ topology of $E^{**}$ to a tripotent denoted by $r(a)$
and called the \emph{range tripotent} of $a$. The tripotent $r(a)$
is the smallest tripotent $e$ in $E^{**}$ satisfying that $a$ is
positive in the JBW$^*$-algebra $E^{**}_{2} (e)$. It is also known
that the sequence $(a^{[2n -1]})$ converges in the
weak$^*$ topology of $E^{**}$ to a tripotent (called the
\hyphenation{support}\emph{support} \emph{tripotent} of $a$)
$s(a)$ in $E^{**}$, which satisfies $ s(a) \leq a \leq r(a)$ in
$A^{**}_2 (r(a))$ (compare \cite[Lemma 3.3]{EdRu88}; the reader should be noted that in \cite{EdRu96}, $r(a)$ is called the support tripotent of
$a$).\smallskip

Two elements $a,b$ in a JB$^*$-triple $E$ are said to be
\emph{orthogonal} (written $a\perp b$) if $L(a,b) =0$. It is known that $a\perp b$ if, and only if, one of the following statements holds:
\begin{center}
\begin{tabular}{ccc}\label{ref orthogo}
  %\hline
  $b\perp a;$ & $\J aab =0;$ & $a \perp r(b);$  \\
  & & \\
  $r(a) \perp r(b);$ & $E^{**}_2(r(a)) \perp E^{**}_2(r(b))$; & $r(a) \in E^{**}_0 (r(b))$; \\
  & & \\
  $a \in E^{**}_0 (r(b))$; & $b \in E^{**}_0 (r(a))$; & $E_a \perp E_b$, \\
\end{tabular}
\end{center}
(see, for example, \cite[Lemma 1]{BurFerGarMarPe}). The natural partial order in the set of tripotents of a JB$^*$-triple $E$ is defined as follows: given two tripotents $e$ and $u$ in $E$ we say that $u \leq e$ if $e-u$ is a tripotent in $E$ and $e-u \perp u$. \smallskip

Range and support tripotent are examples of open and compact tripotents, respectively. In attempt to generalize the studies of C.A. Akemann, L.G. Brown, and G.K. Pedersen on open and compact projections in the bidual of a C$^*$-algebra (\cite{Ak69,Ak70,Ak71,Bro,AP}), C.M. Edwards and G.T. R{\"u}ttimann introduce in \cite{EdRu96} the notions of \emph{open} and \emph{compact} tripotents in the bidual of a JB$^*$-triple. We recall that a tripotent $u$ in the bidual of a JB$^*$-triple $E$ is said to be \emph{open} when $E^{**}_2 (u)\cap E$ is weak$^*$ dense in $E^{**}_2 (u)$. A tripotent $e$ in $E^{**}$ is said to be \emph{compact-$G_{\delta}$} (relative to $E$) if there exists a norm one element $a$ in $E$ such that $e$ coincides with $s(a)$, the support tripotent of $a$. A tripotent $e$ in $E^{**}$ is said to be \emph{compact} (relative to $E$) if there exists a decreasing net $(e_{\lambda})$ of tripotents in $E^{**}$ which are compact-$G_{\delta}$ with infimum $e$, or if $e$ is zero. Closed and bounded tripotents in $E^{**}$ were introduced and studied by the second and third authors of this note in \cite{FerPe06} and \cite{FerPe07}.  A tripotent $e$ in $E^{**}$ such that $E^{**}_0 (e)\cap E$ is weak$^*$ dense in $E^{**}_0(e)$ is called \emph{closed} relative to $E$. When there exists a norm one element $a$ in $E$ such that $a= e + P_0 (e) (a)$, the
tripotent $e$ is called \emph{bounded} (relative to $E$) and we shall write $e\leq_{_T} a$ (c.f. \cite{FerPe06}). The relation
$\leq_{_T}$ is consistent with the natural partial order on the set of tripotents, that is, for any two tripotents $e$ and $u$ we
have $e\leq u$ if and only if $e\leq_{_T} u.$ An useful result established in \cite[Theorem 2.6]{FerPe06} (see also \cite[Theorem 3.2]{FerPe10b}) asserts that a tripotent $e$ in $E^{**}$ is compact if, and only if, $e$ is closed and bounded.\smallskip

Given a JBW$^*$-triple $W,$ a norm one functional $\varphi$ in $W_*$
and a norm one element $z$ in $W$ with $\varphi (z) =1$, Proposition 1.2 in \cite{BarFri} assures that the mapping
$$(x,y)\mapsto \varphi\J xyz$$ is a positive sesquilinear
form on $W.$ Furthermore, for every norm one element $w$ in $W$
satisfying $\varphi (w) =1$, we have $\varphi\J xyz = \varphi\J
xyw,$ for all $x,y\in W$. Thus, the mapping $ x\mapsto \|x\|_{\varphi}:=
\left(\varphi\J xxz\right)^{\frac{1}{2}},$ is a prehilbertian
seminorm on $W$. The strong*-topology %(noted by $S^*(W,W_*)$)
is the topology on $W$ generated by the family $\{
\|\cdot\|_{\varphi}:\varphi\in {W_*}, \|\varphi \| =1 \}$ (c.f.
\cite{BarFri}). For later purposes, we recall that the triple product of a JBW$^*$-triple is jointly
strong$^*$-continuous on bounded sets (see \cite[Theorem 9]{PeRo}). Since the strong*-topology of a JBW$^*$-triple $W$ is compatible with the duality $(W, W_{*})$ (cf. \cite[Corollary 9]{PeRo}), it follows from the bipolar theorem that for each convex $C\subseteq W$ we have %\begin{equation}\label{eq bipolar compatible}
$$\overline{C}^{\sigma(W,W_*)} = \overline{C}^{S^*(W,W_*)}.$$
%\end{equation}
We shall also make use of the following fact due to L. Bunce (see \cite{Bun01}): Let $F$ be a JBW$^*$-subtriple (i.e. a weak$^*$ closed JB$^*$-subtriple) of a JBW$^*$-triple $W$, then the strong$^*$-topology of $F$ coincides with the restriction to $F$ of the strong$^*$-topology of $W$, that is, $S^*(F,F_*) = S^* (W,W_*)|_{F}$.\smallskip

Given a Banach space $X$, we habitually regard $X$ as being contained in
$X^{**}$ and we identify the weak$^*$ closure, in $X^{**}$, of a closed subspace $Y$ of $X$ with $Y^{**}$. Let $F$ be a JB$^*$-subtriple of a JB$^*$-triple $E$ and let $e$ be a tripotent in $F^{**} \equiv \overline{F}^{\sigma(E^{**},E^*)} \subseteq E^{**}$. Corollary 2.9 in \cite{FerPe06} and \cite[Proposition 3.3]{FerPe10b} prove that $u$ is compact in $F^{**}$ if, and only if, $u$ is compact in $E^{**}$.\smallskip

The following lemma, whose statement is required later, can be directly deduced from the joint strong$^*$-continuity of the triple product of every JBW$^*$-triple on bounded sets and the definition of the corresponding Peirce projection.

\begin{lemma}\label{l Peirce zero strong* limits} Let $W$ be a JBW$^*$-triple. Suppose that $(e_\lambda)$ is a net (or a sequence) of tripotents in $W$ converging, in the strong$^*$-topology of $W$ to a tripotent $e$ in $W$. Let $(x_{\mu})$ be a net (or a sequence) in $W$, converging to some $x\in W$ in the strong$^*$-topology. Then, for each $i\in \{0,1,2\}$, the net (sequence) $P_i (e_\lambda) (x_\mu)$ tends to $P_i (e) (x)$. $\hfill\Box$
\end{lemma}

\section{Local triple derivations on JB$^*$-triples are triple derivations}

A triple derivation on a JB$^*$-triple $E$ is a linear mapping $\delta: E\to E$ satisfying that $$\delta \J abc = \J {\delta(a)}bc + \J a{\delta(b)}c + \J ab{\delta(c)},$$ for every $a,b,c\in E$. The Jordan identity implies that, for each $a,b$ in $E$, the mapping $\delta(a,b)= L(a,b)-L(b,a)$ is a triple derivation on $E$. Every triple derivation on a JB$^*$-triple is continuous (cf. \cite[Corollary 2.2]{BarFri} and \cite[Corollary 10]{PeRu}). The separate weak$^*$ continuity of the triple product of $E^{**}$ together with Goldstine's theorem, imply that \begin{equation}\label{eq bitranspose of a derivation} \delta^{**} \hbox{ is a triple derivation on } E^{**} \hbox{ whenever $\delta$ is a triple derivation on $E$.}
\end{equation}

This section contains the main result of the paper, which asserts that every continuous local triple derivation $T$ on a JB$^*$-triple $E$ is a derivation. Our strategy will consist in studying the behavior of $T^{**}$ on compact and range tripotents in $E^{**}$. We start with a technical lemma borrowed from \cite{BurFerGarPe2012}.

\begin{lemma}\label{l 4 BurFerGarPe}\cite[Lemma 4]{BurFerGarPe2012} Let $F$ be a JB$^*$-subtriple of a JB$^*$-triple $E$, where the latter is regarded as a Jordan Banach triple $F$-module with respect to its natural triple product. Let $T: F \to E$ be a local triple derivation. Then the products of the form $\J a{T(b)}c$ vanish for every $a,b,c$ in $F$ with $a, c\perp b$. $\hfill\Box$
\end{lemma}

We shall survey now some of the properties of triple derivations. Let $\delta : E \to E $ be a triple derivation on a JB$^*$-triple. Suppose that $e$ is a tripotent in $E$. In such a case, $$\delta (e) = \delta \J eee = 2 \J {\delta(e)}ee + \J e{\delta(e)}e = 2 P_2 (e) (\delta(e)) +  P_1 (e) (\delta(e)) + Q (e) (\delta(e))$$
$$= 2 P_2 (e) (\delta(e)) +  P_1 (e) (\delta(e)) + \Big(P_2 (e) (\delta(e))\Big)^{*_{e}}, $$ which implies that
\begin{equation}\label{eq triple derivations on tripotents} P_0 (e) (\delta(e))=0 \hbox{ and } P_2 (e) (\delta(e))= - \Big(P_2 (e) (\delta(e))\Big)^{*_{e}} = - Q (e) (\delta(e)).
\end{equation} %In particular, for each triple derivation $\delta$ on a unital C$^*$-algebra $A$, the identities \begin{equation}\label{eq triple derivations on projections} (1-p)\delta(p)(1-p)=0 \hbox{ and } p\delta(p) p= - p \delta(p)^* p,\end{equation} hold for every projection $p\in A.$
If $T : E \to E$ is merely a local triple derivation on $E$, we can find a triple derivation $\delta_e : E \to E$ such that $T(e) = \delta_e (e)$, which gives \begin{equation}\label{eq local triple derivations on tripotents} P_0 (e) (T(e))=0 \hbox{ and } P_2 (e) (T(e))= - \Big(P_2 (e) (T(e))\Big)^{*_{e}} = - Q (e) (T(e)).
\end{equation}

Though a JB$^*$-triple $E$ need not contain, in general, tripotent elements, its bidual has a rich set of tripotents. The next proposition explains the behavior of a continuous local triple derivation on a JB$^*$-triple $E$ on compact tripotents in $E^{**}.$

\begin{proposition}\label{p loal triple derivations on compact tripotents} Let $T : E \to E$ be a bounded local triple derivation on a JB$^*$-triple. Suppose $e$ is a compact tripotent in $E^{**}.$ Then the following statements hold: \begin{enumerate}[$(a)$]
\item $P_0 (e) T^{**}(e) =0$;
\item If $a$ is a norm one element in $E$ whose support tripotent is $e$ (that is, $e$ is a compact-$G_{\delta}$ tripotent), then $Q (e) T(a) = Q (e) T^{**} (e);$
\item $P_2(e) T^{**}(e) = - Q(e) (T^{**}(e)).$
\end{enumerate}
\end{proposition}

\begin{proof}
$(a)$ Let us assume that $e$ is a compact-$G_{\delta}$ tripotent. So, there exists a norm one element $a$ in $E$ such that $s(a) =e$. Let $E_a$ denote the JB$^*$-subtriple of $E$ generated by $a$. We have already mentioned that there exists a subset $L\subseteq (0,1]$ with $1\in \{0\}\cup L$ compact and a triple isomorphism $\Psi$ from $E_a$ onto $C_0(L)$ such that $\Psi (a) (t) =t,$ $\forall t\in L$ (see page \pageref{local Gelfand theory}). Pick a sequence of norm one elements $(b_n)$ in $E_a$ such that $b_n= e + P_0 (e)(b_n),$ $\J {b_{n}}{b_{n+1}}{b_{n}}= \J {b_{n}}{b_{n}}{b_{n+1}} = b_{n+1}$, $(b_n) \to e,$ in the strong$^*$-topology of $E^{**}$ (take, for example, $$b_{n} (t):=\left\{%
\begin{array}{ll}
    0, & \hbox{if $t\in L\cap [0,1-\frac1n ]$;} \\
    n(n+1) (t-1+\frac1n ), & \hbox{if $t\in L\cap [1-\frac1{n},1-\frac1{n+1} ]$;} \\
    1, & \hbox{if $t\in L\cap [1-\frac1{n+1},1 ]$.} \\
\end{array}%
\right).$$

Fix a natural $n.$ Since, the support tripotent of $b_n,$ $s(b_n)$, is a compact tripotent in $E^{**}$, given $z,w\in E_0^{**} (s(b_n))$ we can find (bounded) nets $(c_\mu)$ and $(d_{\nu})$ in $E_0^{**} (s(b_n)) \cap E$ converging to $z$ and $w$ in the strong$^*$-topology of $E^{**},$ respectively. Clearly, $c_\mu, d_{\nu} \perp b_{n+1}$ for every $\mu$ and $\nu$, and hence, by Lemma \ref{l 4 BurFerGarPe}, $$\J {c_\mu}{T(b_{n+1})}{d_\nu} =0 \ \ (\forall \mu, \nu).$$ Taking strong$^*$-limits in $\mu$ and $\nu$ we have $$\J {z}{T(b_{n+1})}{w} =0, \hbox{ for every } n\in \mathbb{N}, z,w\in E_0^{**} (s(b_n)),$$ equivalently, $$\J {P_0^{**} (s(b_n)) (x)}{T(b_{n+1})}{P_0^{**} (s(b_n))(y)} =0, \hbox{ for every } n\in \mathbb{N}, x,y\in E^{**}.$$

Now, since the triple product of $E^{**}$ is jointly strong$^*$-continuous and $T^{**}$ is strong$^*$-continuous, we can take strong$^*$-limit in the above expression, and by Lemma \ref{l Peirce zero strong* limits}, we have $\J {P_0^{**} (e)(x)}{T^{**}(e)}{P_0^{**} (e)(y)}=0,$ for every $x,y\in E^{**}$. It follows, for example, from Peirce arithmetic and the third axiom in the definition of JB$^*$-triples, that $P_0^{**} (e) T^{**}(e)=0$.\smallskip

Let us consider a compact tripotent $e\in E^{**}.$ By definition, there exists a decreasing
net $(e_{\lambda})$ of compact-$G_{\delta}$ tripotents in  $E^{**}$ converging to $e$ in the strong$^*$-topology of $E^{**}$. From the above argument, $P_0 (e_{\lambda}) T^{**}(e_{\lambda}) =0$ ($\forall \lambda$), and by Lemma \ref{l Peirce zero strong* limits}, $P_0 (e) T^{**}(e) =0$, as we desired.\smallskip

$(b)$ Again, let $a$ be a norm one element in $E$ such that $s(a) = e $. Let us denote $a_0 = P_0 (e)(a)$. Adapting the previous arguments, we consider the JB$^*$-subtriple, $E_a,$ generated by $a$, and pick two sequences $(a_n)$ and $(b_n)$ in the closed unit ball of $E_a$ defined by $$a_{n} (t):=\left\{%
\begin{array}{ll}
    t, & \hbox{if $t\in L\cap [0,1-\frac1n]$;} \\
    - n(n+1) (t-1+\frac1{n+1} ), & \hbox{if $t\in L\cap [1-\frac1{n},1-\frac1{n+1} ]$;} \\
    0, & \hbox{if $t\in L\cap [1-\frac1{n+1},1 ]$} \\
\end{array}%
\right. ,$$ and $$b_{n} (t):=\left\{%
\begin{array}{ll}
    0, & \hbox{if $t\in L\cap [0, 1-\frac1{n+1}]$;} \\
    (n+1) (t-1+\frac1{n+1}), & \hbox{if $t\in L\cap [1-\frac1{n+1},1 ]$.} \\
\end{array}%
\right.$$ Clearly, $a_n\perp b_n$ ($\forall n$), $(a_n)\to a_0$ and $(b_n) \to e$ in the strong$^*$-topology of $E^{**}.$ Lemma \ref{l 4 BurFerGarPe} assures that $$\J {b_n}{T(a_n)}{b_n} = 0 \ \ (\forall n\in \mathbb{N}).$$ Taking strong$^*$ limits in the above expression we have $\J e{T^{**} (a_0)}e =0$, and hence $\J e{T^{**} (a)}e = \J e{T^{**} (e)}e.$\smallskip

$(c)$ Assume, one more time, that $e$ is a compact-$G_{\delta}$ tripotent and $a$ is a norm one element in $E$ such that $s(a) = e $. Since $T$ is a local triple derivation, we can find a triple derivation $\delta_{a}: E\to E$ such that $T(a) = \delta_{a} (a).$ We notice that $\delta_a^{**} : E^{**}\to E^{**}$ is a triple derivation on $E^{**}$ (see \eqref{eq bitranspose of a derivation}). Since $\delta_a$ is triple derivation, the identity in $(b)$ also holds whenever we replace $T$ with $\delta_a.$ Therefore, $$P_2(e) T^{**} (e) =  P_2 (e) T(a) = P_2(e) \delta_a (a) = P_2 (e) \delta_a^{**} (e) $$
$$\hbox{(by $(\ref{eq triple derivations on tripotents})$)} = - Q(e) \Big(\delta_a^{**} (e)\Big) = - Q(e) \Big(\delta_a (a)\Big)  = - Q(e) \Big(T (a)\Big) = - Q(e) \Big(T^{**} (e) \Big),$$ which proves statement $(c)$ for compact-$G_{\delta}$ tripotents in $E^{**}$.\smallskip

Let us consider a decreasing net $(e_{\lambda})$ of compact-$G_{\delta}$ tripotents in $E^{**}$ converging to $e$ in the strong$^*$-topology of $E^{**}$. Since, for each $\lambda$, $P_2(e_{ \lambda}) T^{**} (e_{\lambda}) = - Q(e_{\lambda}) (T^{**}(e_{\lambda}))$, Lemma \ref{l Peirce zero strong* limits}, assures the desired identity for $e$.\end{proof}

In the hypothesis of the above proposition, let $a$ be a norm one element in $E$ and let $E_a$ denote the JB$^*$-subtriple of $E$ generated by $a$. By the Gelfand theory for commutative JB$^*$-triples, there exists a subset $L\subseteq (0,1]$ with $1\in \{0\}\cup L$ compact and a triple isomorphism $\Psi$ from $E_a$ onto $C_0(L)$ such that $\Psi (a) (t) =t$ ($\forall t\in L$). Clearly, the range tripotent of $a$ can be approximated, in the strong$^*$ topology of $E^{**}$, by a sequence $(e_n)$ of compact-$G_{\delta}$ tripotents in $E^{**}$, that is, $(e_n) \to r(a)$ in the strong$^*$-topology. Since, by the above Proposition \ref{p loal triple derivations on compact tripotents}, $$P_0 (e_n) T^{**}(e_n) =0,$$ and $$P_2(e_n) T^{**}(e_n) = - Q(e_n) (T^{**}(e_n)),$$ taking strong$^*$-limit in $n$ we deduce, by Lemma \ref{l Peirce zero strong* limits}, that \begin{equation}\label{eq range tripotents in bidual}P_0 (r(a)) T^{**}(r(a)) =0, \hbox{ and, } P_2(r(a)) T^{**}(r(a)) = - Q(r(a)) (T^{**}(r(a))).
\end{equation}

Let $e$ be a compact or a range tripotent in $E^{**}$. It follows from Proposition \ref{p loal triple derivations on compact tripotents} and ($\ref{eq range tripotents in bidual}$) that $$T^{**} \J eee = T^{**} (e) = P_2(e) T^{**} (e) + P_1 (e) T^{**} (e),$$
$$2 \J ee{T^{**} (e)} = 2 P_2(e) T^{**} (e) + P_1 (e) T^{**} (e),$$ and $$\J e{T^{**}(e)}e = Q(e) T^{**} (e) = -P_2 (e) T^{**} (e),$$ which assures that \begin{equation}
\label{eq T on compact and range tripotents is a triple derivation} T^{**} \J eee = 2 \J ee{T^{**} (e)} + \J e{T^{**}(e)}e.
\end{equation}

\begin{corollary}\label{c bitranspose compact range tripotents} Let $T : E \to E$ be a continuous local triple derivation on a JB$^*$-triple. Suppose that $e$ and $u$ are two orthogonal compact tripotents in $E^{**}$, $r_1$ and $r_2$ are two orthogonal range tripotent in $E^{**}$ and $e\perp r_2$. Then the following identities hold: $$T^{**} \J {e\pm u}{e\pm u}{e\pm u} = 2 \J {e\pm u}{e\pm u}{T^{**} (e\pm u)}+ \J {e\pm u}{T^{**}(e\pm u)}{e\pm u},$$ $$T^{**} \J {r_1\pm r_2}{r_1\pm r_2}{r_1\pm r_2} = 2 \J {r_1\pm r_2}{r_1\pm r_2}{T^{**} (r_1\pm r_2)}$$ $$ + \J {r_1\pm r_2}{T^{**}(r_1\pm r_2)}{r_1\pm r_2};$$ and
 $$T^{**} \J {e\pm r_2}{e\pm r_2}{e\pm r_2} = 2 \J {e\pm r_2}{e\pm r_2}{T^{**} (e\pm r_2)}$$ $$ + \J {e\pm r_2}{T^{**}(e\pm r_2)}{e\pm r_2}.$$
\end{corollary}

\begin{proof} Since $e$ and $u$ are two orthogonal compact tripotents in $E^{**},$ it follows from Proposition 3.7 in \cite{FerPe10b} that $e\pm u$ is a compact tripotent in $E^{**}.$ It is easy to see that the sum and the difference of two orthogonal range tripotents in $E^{**}$ is again a range tripotent in $E^{**}$. Thus, the first and the second identity have been proved in ($\ref{eq T on compact and range tripotents is a triple derivation}$).\smallskip

To see the last identity, we recall that since $e\perp r_2$ and $r_2$ is a range projection, we can find a sequence of compact tripotents $(e_n)$ in $E^{**}$ such that $e_n\leq r_2$, and hence $e_n\perp e$ for every $n$, and $(e_n)\to r_2$ in the strong$^*$-topology of $E^{**}.$ The first identity shows that $$T^{**} \J {e\pm e_{n}}{e\pm e_{n}}{e\pm e_{n}} = 2 \J {e\pm e_{n}}{e\pm e_{n}}{T^{**} (e\pm e_{n})}$$ $$+ \J {e\pm e_{n}}{T^{**}(e\pm e_{n})}{e\pm e_{n}},$$ for every $n$. The desired equality follows by taking strong$^*$-limit in $n$.
\end{proof}

Let $T : E \to E$ be a bounded local triple derivation on a JB$^*$-triple. Another application of the local Gelfand structure of JB$^*$-triples allows us to see that each element $a$ in a JB$^*$-triple $E$ can be approximated in norm by a finite real linear combination of mutually orthogonal range and compact tripotents in $E_a^{**}\subseteq E^{**}$. Corollary \ref{c bitranspose compact range tripotents} will imply that $T^{**}$ behaves as a triple derivation on those elements which are finite real linear combinations of mutually orthogonal range and compact tripotents in $E^{**},$ and consequently, $T$ (and hence $T^{**}$) is a triple derivation.

%We can therefore assure, thanks to Corollary \ref{c bitranspose compact range tripotents} that the arguments given in \cite[Lemma 5.5 to Theorem 5.11]{Mack} remain valid for finite families of mutually orthogonal range and compact tripotents in $E^{**}$. Consequently  We shall include here an sketch of the proof for completeness reasons.

We can state now our main result, which solves a problem conjectured by M. Mackey in \cite[Conjecture (C3)]{Mack} (and also posed in \cite[Problem 1]{BurFerGarPe2012}).

\begin{theorem}\label{thm local triple derivations are triple derivations} Every bounded local triple derivation on a JB$^*$-triple is a triple derivation.
\end{theorem}

\begin{proof} Let $T: E \to E$ be a bounded local triple derivation on a JB$^*$-triple $E$. Let $e_1 , \ldots, e_m$ be a family of mutually orthogonal range or compact tripotents in $E^{**}$. Let us pick $i,j,k\in \{1,\ldots, m\}$ with $i,k \neq j.$ By Proposition \ref{p loal triple derivations on compact tripotents} and $(\ref{eq range tripotents in bidual})$ we have $P_0 (e_j) T^{**} (e_j)=0.$ By assumptions, $e_i,e_k\in E^{**}_0 (e_j)$ and hence \begin{equation}\label{eq 1 theorem main} \J {e_i}{T^{**}(e_j)}{e_k} = 0,
\end{equation} by Peirce arithmetic.

Now, fix $i\neq j$ in $\{1,\ldots, m\}$. Since $e_i$ and $e_j$ are compact or range tripotents in $E^{**}$, Corollary \ref{c bitranspose compact range tripotents} implies that \begin{equation}\label{eq 1b theorem main} T^{**} \J {e_i }{e_i}{e_i} = 2 \J {e_i }{e_i }{T^{**} (e_i)} + \J {e_i }{T^{**}(e_i)}{e_i};
\end{equation} and $$ T^{**} \J {e_i \pm e_j}{e_i \pm e_j}{e_i \pm e_j} = 2 \J {e_i \pm e_j}{e_i \pm e_j}{T^{**} (e_i \pm e_j)}$$ $$ + \J {e_i \pm e_j}{T^{**}(e_i \pm e_j)}{e_i \pm e_j}.$$ Combining the last two identities we get $$\pm 2 \J {e_i}{e_i}{T^{**} (e_j)} + 2 \J {e_j}{e_j}{T^{**} (e_i)} \pm  \J {e_i }{T^{**}(e_j)}{e_i )} + \J {e_j}{T^{**}(e_i)}{ e_j} $$ $$\pm 2 \J {e_i}{T^{**}(e_i)}{ e_j} + 2 \J {e_i}{T^{**}( e_j)}{e_j} =0, $$ and consequently, $$+ 4 \J {e_j}{e_j}{T^{**} (e_i)} + 2 \J {e_j}{T^{**}(e_i)}{ e_j}+ 4 \J {e_i}{T^{**}( e_j)}{e_j} =0.$$ Applying ($\ref{eq 1 theorem main}$) we obtain \begin{equation}\label{eq 2 theorem main} \J {e_j}{e_j}{T^{**} (e_i)} + \J {e_i}{T^{**}( e_j)}{e_j} =0.
\end{equation}

Consider now an element $\displaystyle b=\sum_{i=1}^{m} \lambda_i e_i,$ where $e_1 , \ldots, e_m$ are as above. Having in mind that  $e_1 , \ldots, e_m$ are mutually orthogonal, we compute \begin{equation}\label{eq 3 theorem main} T^{**}(\J bbb) = \sum_{i=1}^{m} \lambda_i^{3} T^{**} (\J {e_i}{e_i}{e_i});
\end{equation}
\begin{equation}\label{eq 4 theorem main} 2 \J {T^{**}(b)}bb =2 \sum_{i,j=1}^{m} \lambda_i^2 \lambda_j \J {e_i}{e_i}{T^{**} (e_j)}
\end{equation} $$= 2\sum_{i=1}^{m} \lambda_i^3 \J {e_i}{e_i}{T^{**} (e_i)} + 2\sum_{i,j=1, i\neq j}^{m} \lambda_i^2 \lambda_j \J {e_i}{e_i}{T^{**} (e_j)};$$
\begin{equation}\label{eq 5 theorem main} \J b{T^{**}(b)}b = \hbox{ (by ($\ref{eq 1 theorem main}$)) } = \sum_{i=1}^{m} \lambda_i^3 \J {e_i}{T^{**} (e_i)}{e_i} +2 \sum_{i,j=1, i\neq j}^{m} \lambda_i^2 \lambda_j \J {e_i}{T^{**} (e_i)}{e_j}.
\end{equation} Finally, by $(\ref{eq 1b theorem main})$ and $(\ref{eq 2 theorem main})$ we have $$T^{**} \J bbb = 2 \J {T^{**}(b)}bb + \J b{T^{**}(b)}b.$$ Since every element $a\in E$ can be approximated in norm, by elements of the form $\displaystyle b=\sum_{i=1}^{m} \lambda_i e_i$, we conclude that $T\J aaa = 2 \J {T^{**}(a)}aa + \J a{T^{**}(a)}a$, for every $a\in E.$ Finally, an standard polarisation argument, via formula \begin{equation}
\label{polarization fla} \{ x,y,z\} = 8^{-1} \sum_{k=0}^{3} \sum_{j=1}^{2} (-1)^{j} i^{k} \Big(x+ i^{k} \ y + (-1)^{j} z\Big)^{[3]} \ \ \
(x,y, z\in {E}),
\end{equation} assures that $T$ is a triple derivation.
\end{proof}

In 1997, J.M. Isidro, W. Kaup and A. Rodríguez Palacios introduce a strictly wider class of Jordan triples over the real field and called the elements of this new category \emph{real JB$^*$-triples}. A \emph{real JB$^*$-triple} is a norm closed real subtriple of a JB$^*$-triple (cf. \cite{IsKaRo95}). When restricted to real scalar multiplication, every (complex) JB$^*$-triple is also a real JB$^*$-triple. The class of real JB$^*$-triples also includes real C$^*$-algebras, JB-algebras, JC-algebras, real JB$^*$-algebras, operator spaces between real, complex and quaternionic Hilbert spaces and real Hilbert spaces. Every real JB$^*$-triple can be complexified to become a JB$^*$-triple. Furthermore, every real JB$^*$-triple $A$ is a real form of a complex
JB$^*$-triple, that is, there exist a (complex) JB$^*$-triple $B\cong A\oplus i A$ and a period 2 conjugate linear isometry  $\tau: B\rightarrow B$ (which is also a JB$^*$-triple homomorphism) such that $A=\{b\in B\::\:\tau(b)=b\}$ (see \cite{IsKaRo95}). Let us consider $\widetilde{\tau} : B^{*} \rightarrow B^{*}$ defined by
$\widetilde{\tau} (\phi) (z) = \overline{\phi (\tau (z))}.$
The mapping $\widetilde{\tau}$ is a conjugation on $B^{*}$, and the restriction mapping $$(B^{*})^{\widetilde{\tau}}
\longrightarrow (B^{\tau})^{*}\ (\cong A^*)$$ $$\phi \mapsto \phi|_{A}$$ is an
isometric bijection, where $(B^{*})^{\widetilde{\tau}} := \{ \phi \in
B^{*} : \widetilde{\tau} (\phi)= \phi \}.$ The second transpose $\tau^{**} : B^{**}\to B^{**}$ is a period 2 conjugate linear isometry satisfying $\left(B^{**}\right)^{\tau^{**}} = A^{**}$. \smallskip

A real \emph{JBW$^*$-triple} is a real JB$^*$-triple which is also a dual Banach space. Every real JBW$^*$-triple has a unique (isometric) predual and separately weak$^*$ continuous triple product (see \cite{MarPe}). The bidual of every real JB$^*$-triple is a JBW$^*$-triple (compare \cite{IsKaRo95}).\smallskip

The notions of range, support or compact-$G_{\delta}$, open, closed and compact tripotents also make sense in the bidual of every real JB$^*$-triple. When real JB$^*$-triples are regarded as real forms of complex JB$^*$-triples, the generalised Urysohn's lemmas proved by the second and third author of this note in \cite{FerPe10b} remain valid for real JB$^*$-triples. Furthermore, an appropriate local Gelfand theory for single-generated real JB$^*$-subtriples is also available in the real setting (cf. \cite[\S 3]{BurPeRaRu}). Thus, the arguments given above to prove Theorem \ref{thm local triple derivations are triple derivations} can be applied to show that every bounded local triple derivation $T$ on a real JB$^*$-triple $A$ satisfies that $T\J aaa = 2\J {T(a)}aa + \J a{T(a)}a$, for every $a\in A.$ Unfortunately, the polarisation formula \eqref{polarization fla} employed at the end of the proof of Theorem \ref{thm local triple derivations are triple derivations} is not valid for real JB$^*$-triples, so we  cannot  obtain the conclusion of that theorem in the real setting.

\begin{corollary}\label{c local derivations on real JB$^*$-triples} Let $T$ be a continuous local triple derivation on a real JB$^*$-triple $A$. Then $T\J aaa = 2\J {T(a)}aa + \J a{T(a)}a$, for every $a\in A,$ that is, $T$ is a triple derivation of the symmetrized Jordan triple product $3 <\!a,b,c\!>:= \J abc + \J cab + \J bca.$  $\hfill\Box$
\end{corollary}

In the light of the above corollary, it seems natural to consider the following problem:

\begin{problem}\label{problem local triple derivations on real JB*-triples} Is every bounded local triple derivation on a real JB$^*$-triple a triple derivation?
\end{problem}

In a recent paper (see \cite{PeRu}), B. Russo and the third author of this note initiated the study of triple module-valued triple derivations on (real and complex) JB$^*$-triples. Let $E$ be a complex (resp. real) Jordan triple. We recall that a \emph{Jordan triple $E$-module}  (also called \emph{triple $E$-module}) is a
vector space $X$ equipped with three mappings $$\{.,.,.\}_1 :
X\times E\times E \to X, \quad \{.,.,.\}_2 : E\times X\times E \to
X$$ $$ \hbox{ and } \{.,.,.\}_3: E\times E\times X \to X$$
satisfying  the following axioms:
\begin{enumerate}[{$(JTM1)$}]
\item $\{ x,a,b \}_1$ is linear in $a$ and $x$ and conjugate
linear in $b$ (resp., trilinear), $\{ a,b,x \}_3$ is linear in $b$
and $x$ and conjugate linear in $a$ (resp., trilinear) and
$\{a,x,b\}_2$ is conjugate linear in $a,b,x$ (resp., trilinear)
\item  $\{ x,b,a \}_1 = \{ a,b,x \}_3$, and $\{ a,x,b \}_2  = \{
b,x,a \}_2$  for every $a,b\in E$ and $x\in X$. \item Denoting by
$\J ...$ any of the products $\{ .,.,. \}_1$, $\{ .,.,. \}_2$ and
$\{ .,.,. \}_3$, the Jordan identity \eqref{Jordan  identity}
holds whenever one of the elements is in $X$ and the rest are in $E$.
\end{enumerate} When the products $\{ .,.,.
\}_1$, $\{ .,.,. \}_2$ and $\{ .,.,. \}_3$ are (jointly)
continuous we shall say that $X$ is a \emph{Banach (Jordan) triple
$E$-module}. Henceforth, the triple products $\J \cdot\cdot\cdot_j$, $ j=1,2,3$, which occur in the definition of Jordan triple module will be denoted simply by $\J \cdot\cdot\cdot$ whenever the
meaning is clear from the context.\smallskip

It is obvious that every real or complex Jordan triple $E$ is a {\it real} triple $E$-module. The dual space,
$E^*$, of a complex (resp., real) Jordan Banach triple $E$ is a complex (resp., real) triple $E$-module with respect to the products: $$  \J ab{\varphi} (x) = \J {\varphi}ba (x) := \varphi \J bax, \hbox{ and, } \J a{\varphi}b (x) := \overline{ \varphi \J axb },$$ for all $\varphi\in E^*$, $a,b,x\in E$.\smallskip

Let $E$ be a complex (respectively, real) JB$^*$-triple and let $X$ be a
triple $E$-module. We recall that a conjugate linear (respectively, linear)
mapping $\delta : E \to X$ is said to be a \emph{triple or ternary derivation} if
$$\delta \{a,b,c\} = \J {\delta(a)}bc +\J a{\delta(b)}c + \J ab{\delta(c)}.$$ A conjugate linear (respectively, linear) mapping $T: E\to X$ will be called a \emph{local triple derivation} if for each $a$ in $E$ there exists a triple derivation $\delta_{a}: E \to X$ such that $T(a)= \delta_a (a).$

\begin{problem}\label{p local triple derivations onto triple modules} Is every continuous local triple derivation from a real or complex JB$^*$-triple $E$ into a Banach triple $E$-module a triple derivation?
\end{problem}

\subsection{Automatic continuity}

We establish now an automatic continuity result for local triple derivations, giving a positive answer to Problem \eqref{problem local triple derivations continuous}. We shall review the arguments given by T. Barton and Y. Friedman to show that a triple derivation on a JB$^*$-triple is a triple derivation. Let $X$ be a complex Banach space. We recall that a linear mapping $T: X\to X$ is \emph{dissipative} if for each $x\in X$ and each functional $\phi\in X^*$ with $\|x\|= \|\phi\| = \phi (x) =1$ we have $\Re\hbox{e} \phi (T(x)) \leq 0$. It is known that $T$ is continuous whenever it is dissipative (compare \cite[Proposition 3.1.15]{BraRo}). In \cite[Theorem 2.1]{BarFri}, Barton and Friedman prove that every derivation on a JB$^*$-triple $E$ is dissipative and hence continuous. Let us review some aspect in their proof. Let $x$ be an element in $E$, let $\phi$ a functional in $E^*$ with $\|x\|= \|\phi\| = \phi (x) =1$ and let $\delta : E \to E$ a triple derivation. Applying Peirce arithmetic and support tripotents, the arguments in the proof of Theorem 2.1 in \cite{BarFri} show that $\phi \delta (\J xxx - x) = 0$ and hence $\Re\hbox{e} \phi (\delta(x)) = 0$. It is also justified in the same result that $\phi \J axx = \phi (a)$ and $\phi \J xax = \overline{\phi (a)}$, for every $a$ in $E$. Let $T: E \to E$ be a local triple derivation, not assumed a priori to be continuous, on $E$. Pick a triple derivation $\delta_x$ satisfying $\delta_x (x) = T(x)$. In this case, we have $$2 \phi T(x) + \overline{\phi T(x)} = 2\phi \J {T(x)}xx  +  \phi \J x{T(x)}x $$
$$ = 2\phi \J {\delta_x(x)}xx  +  \phi \J x{\delta_x(x)}x = \phi \delta_x \J xxx .$$ It follows from the above that $$2 \Re\hbox{e} \phi T(x) = \phi \delta_x \J xxx - \phi T(x) =  \phi \delta_x \J xxx - \phi \delta_x (x) =  \phi \delta_x (\J xxx - x) =0,$$ because $\delta_x$ is a triple derivation on $E$. This shows that $T$ is dissipative.

\begin{theorem}\label{t automatic continuity} Every local triple derivation on a (real or complex) JB$^*$-triple is continuous. Consequently, every local triple derivation on a JB$^*$-triple is a triple derivation.$\hfill\Box$
\end{theorem}

\section{Local triple derivations and generalised derivations on C$^*$-algebras}

We begin this section with the following corollary which solves Problem 2 in \cite{BurFerGarPe2012}.

\begin{theorem}\label{thm local triple derivations on C*-algebras are triple derivations} Every local triple derivation on a C$^*$-algebra is a triple derivation. $\hfill\Box$
\end{theorem}

In \cite{BurFerGarPe2012}, it is established that every local triple derivation on a unital C$^*$-algebra is a triple derivation. The strategy to prove this result relies on the connections between (local) triple derivations and generalised derivations on unital C$^*$-algebras in the sense introduced by J. Li and Z. Pan in \cite{LiPan}. We recall that a linear mapping $D$ from a unital C$^*$-algebra $B$ to a (unital) Banach $B$-bimodule $X$ is called a \emph{generalised derivation} whenever the identity $$D (ab) = D(a) b + a D(b) - a D(1) b$$ holds for every $a,b$ in $B$. A \emph{generalised Jordan derivation} from $B$ to $X$ is a linear mapping $D: B \to X$ satisfying $D (a\circ b) = D(a)\circ b + a\circ D(b) - U_{a,b} D(1)$, for every $a,b$ in $B$, where the Jordan product is given by $a\circ b := \frac12 (a b + ba)$ and $U_{a,b} (x) := \frac12 ( axb + bxa) $ (cf. \cite{BurFerGarPe2012}). Every generalised (Jordan) derivation $D : B\to X$  with $D(1) =0$ is a (Jordan) derivation. As remarked in \cite[Remark 8]{BurFerGarPe2012}, generalised Jordan derivations from $B$ to $X$ and a generalised derivations from $B$ to $X$ coincide. A linear mapping $T : B \to X$ is a \emph{local generalised (Jordan) derivation} if for each $a\in B$, there exists a generalised (Jordan) derivation $D_a : B \to X$ with $T(a) = D_a (a)$.\smallskip

In the setting of unital C$^*$-algebras, J. Alaminos, M. Bresar, J. Extremera, and A. Villena \cite[Corollary 3.2]{AlBreExVill} and J. Li and Z. Pan \cite[Corollary 2.9]{LiPan} established the following interesting result:

\begin{theorem}\label{t LiPan}{\rm (\cite[Corollary 3.2]{AlBreExVill}, \cite[Corollary 2.9]{LiPan})} Suppose that $B$ is a unital C$^*$-algebra and $M$ is a unital Banach $B$-bimodule. Then for any
bounded linear map $T$ from $B$ to $M$, the following are equivalent:\begin{enumerate}[$(i)$] \item $T$ is a generalised (Jordan) derivation from $B$ to $M$;
\item $T$ is a local generalised (Jordan) derivation from $B$ to $M$;
\item $aT(b)c = 0$, whenever $a, b, c \in  B$ with $ab = bc = 0$.$\hfill\Box$
\end{enumerate}
\end{theorem}

Unfortunately, the above definitions of generalised (Jordan) derivations and local generalised (Jordan) derivation make sense only when the domain is a unital C$^*$-algebra. However, statement $(iii)$ in the above theorem makes sense in the non-unital setting too. Our next goal is to generalise the above theorem to the setting of not necessarily unital C$^*$-algebras. In a first step we should consider a consequence derived from Theorem 4.5 in \cite{AlBreExVill09}. First we recall a definition taken from \cite[\S 4]{AlBreExVill09}: a generalised derivation from a Banach algebra $A$ to a Banach $A$-bimodule $X$ is a linear operator $D: A \to X$ for which there exists $\xi\in  X^{**}$ satisfying $$D(ab) = D(a)  b + a  D(b) - a \xi b \hbox{ ($a, b \in A$).}$$

\begin{theorem}\label{t AlBreExVill}{\cite[Theorem 4.5]{AlBreExVill09}} Suppose that $A$ is a C$^*$-algebra and $X$ is an essential Banach $A$-bimodule. Let $T: A \to X$ be a continuous linear operator
satisfying $$a,b,c \in A, ab = bc = 0 \Rightarrow aT(b)c = 0.$$ Then $T$ is a generalized derivation.
\end{theorem}

Given a C$^*$-algebra, $A$, the
\emph{multiplier algebra} of $A$, $M(A)$, is the set of all
elements $x\in A^{**}$ such that, for each element $a\in A$, $x
a$ and $a x$ both lie in $A$. We notice that $M(A)$ is a C$^*$-algebra
and contains the unit element of $A^{**}$. Clearly, $A= M(A)$ whenever $A$ is unital.

For later purposes, we recall some basic results on Arens regularity (cf. \cite{Arens51}). Let $X$, $Y$ and $Z$ be Banach spaces and let $m : X\times Y \to Z$ be a bounded bilinear mapping. Defining $m^* (z^\prime,x) (y) := z^\prime (m(x,y))$ $(x\in X, y\in Y, z^\prime\in Z^*)$, we obtain a bounded bilinear mapping $m^*: Z^*\times X \to Y^*.$ Iterating the process, we define a mapping $m^{***}: X^{**}\times Y^{**} \to Z^{**}.$ The mapping $x^{\prime\prime}\mapsto  m^{***}(x^{\prime\prime} , y^{\prime\prime})$ is weak$^*$ to weak$^*$ continuous whenever we fix  $y^{\prime\prime} \in  Y^{**}$, and the mapping $y^{\prime\prime}\mapsto  m^{***}(x, y^{\prime\prime})$ is weak$^*$ to weak$^*$ continuous for every $x\in  X$. One can consider the transposed mapping $m^{t} : Y\times X\to Z,$ $m^{t} (y,x) = m(x,y)$ and the extended mapping $m^{t***t}: X^{**}\times Y^{**} \to Z^{**}.$ In this case, the mapping $x^{\prime\prime}\mapsto  m^{t***t}(x^{\prime\prime} , y)$ is weak$^*$ to weak$^*$ continuous whenever we fix  $y \in  Y$, and the mapping $y^{\prime\prime}\mapsto  m^{t***t}(x^{\prime\prime}, y^{\prime\prime})$ is weak$^*$ to weak$^*$ continuous for every $x^{\prime\prime}\in  X^{**}$.

In general, the mappings $m^{t***t}$ and $m^{***}$ do not coincide (cf. \cite{Arens51}). When $m^{t***t}=m^{***},$
we say that $m$ is Arens regular. It is well known that the product of every C$^*$-algebra is Arens regular and the unique Arens extension of the product of $A$ to $A^{**}\times A^{**}$ coincides with the product of its enveloping von Neumann algebra (cf. \cite[Corollary 3.2.37]{Dales00}).

Let $X$ be a Banach $A$-bimodule over a C$^*$-algebra $A.$ Let us denote by $\pi_1: A\times X \to X$ and $\pi_2: X\times A \to X$ the corresponding module operations given by $\pi_1(a,x) = a x$ and $\pi_2(x,a) = x a$, respectively. By an abuse of notation, given $a\in A^{**}$ and $z\in X^{**},$ we shall frequently  write $a z = \pi_1^{***} (a,z)$ and $z a = \pi_2^{***} (z,a)$. It is known that $X^{**}$ is a Banach $A^{**}$-bimodule for the just defined operations (\cite[Theorem 2.6.15$(iii)$]{Dales00}). It is also known that whenever $(a_\lambda)$ and $(x_\mu)$ are nets in $A$ and $X$, respectively, such that $a_\lambda \to a\in A^{**}$ in the weak$^*$ topology of $A^{**}$ and $x_\mu\to x\in X^{**}$ in the weak$^*$ topology of $X^{**}$, then \begin{equation}\label{eq product bidual module} a x = \pi_1^{***}(a,x) = \lim_{\lambda} \lim_{\mu} a_\lambda x_\mu \hbox{ and } x a= \pi_2^{***} (x,a) =  \lim_{\mu} \lim_{\lambda}  x_\mu a_\lambda
 \end{equation} in the weak$^*$ topology of $X^{**}$ (cf. \cite[2.6.26]{Dales00}).\smallskip

Our next proposition completes the whole picture.

\begin{proposition}\label{p multipliers 1} Let $X$ be an essential Banach $A$-bimodule over a C$^*$-algebra $A$ and let $T: A \to X$ be a bounded linear operator. The following are equivalent:\begin{enumerate}[$(a)$] \item $T^{**}|_{M(A)} : M(A) \to X^{**}$ is a generalised (Jordan) derivation;
\item $T^{**} : A^{**} \to X^{**}$ is a generalised (Jordan) derivation;
\item $T^{**} = d + M_{T^{**} (1)}$, where $d: A^{**} \to X^{**}$ is a derivation and $M_{T^{**}(1)} (a) = T^{**} (1) \circ a = \frac12 ( a T^{**} (1) + T^{**} (1) a)$;
\item $T^{**} : A^{**} \to X^{**}$ is a local generalised (Jordan) derivation;
\item $T^{**}|_{M(A)} : M(A) \to X^{**}$ is a local generalised (Jordan) derivation;
\item $T$ is a generalised derivation;
\item $a T(b) c = 0$, whenever $ab=bc=0$ in $A$.
\end{enumerate}
\end{proposition}

\begin{proof} The implications $(b)\Leftrightarrow (c) \Rightarrow (d)\Rightarrow (e)\Rightarrow (g)$ are clear. The equivalence $(f)\Leftrightarrow (g)$ was established in Theorem \ref{t AlBreExVill} (\cite[Theorem 4.5]{AlBreExVill09}). To see $(a)\Rightarrow (b)$ suppose that $T^{**}|_{M(A)}$ is a generalised derivation, i.e., $$T(a b) = T(a) b + a T(b) - a T^{**} (1) b,$$ for every $a,b$ in $M(A)$. Since, by Goldstine's theorem, the closed unit ball of $A$ is weak$^*$ dense in the closed unit ball of $A^{**}$, we deduce from \eqref{eq product bidual module} that the above equality also holds when $a$ and $b$ are in $A^{**}$ and $T$ is replaced with $T^{**}.$ Thus, $T^{**}$ is a generalised derivation.\smallskip

We shall finally prove the implication $(g)\Rightarrow (a)$. Let $a, b, c$ be elements in $M(A)$ with $a b= bc =0$. We may assume that $a,b$ and $c$ lie in the closed unit ball of $M(A)$. We observe that $a^{[2n-1]} b= 0$ for every natural $n$. Therefore, $\alpha b= 0$, for every $\alpha$ in the JB$^*$-subtriple, $M(A)_a,$ of $M(A)$ generated by $a$. We can similarly show that \begin{equation}\label{eq 3 prop multipliers}\alpha \beta = \beta \gamma = 0
\end{equation} for every $\alpha\in M(A)_{a}$, $\beta\in M(A)_{b}$ and $\gamma\in M(A)_{c}$.

Since $M(A)$ is a C$^*$-subalgebra of $A^{**}$, by Goldstine's theorem, we can find nets $(x_{\lambda})$, $(y_{\mu})$ and $(z_{\nu})$ in the closed unit ball of $A$, converging in the weak$^*$ topology of $A^{**}$ to $a^{[\frac13]}$, $b^{[\frac13]}$ and $c^{[\frac13]}$, respectively. The nets $\left(a^{[\frac13]} x_{\lambda}^* a^{[\frac13]}\right)$, $ \left( b^{[\frac13]} y_{\mu}^* b^{[\frac13]}\right) $, and $\left(c^{[\frac13]} z_{\nu}^* c^{[\frac13]}\right)$ lie in $A$, and by $(\ref{eq 3 prop multipliers})$, we have $$\left(a^{[\frac13]} x_{\lambda}^* a^{[\frac13]}\right) \left( b^{[\frac13]} y_{\mu}^* b^{[\frac13]}\right) =0 = \left( b^{[\frac13]} y_{\mu}^* b^{[\frac13]}\right) \left(c^{[\frac13]} z_{\nu}^* c^{[\frac13]}\right),$$ for every $\lambda, \mu$ and $\nu$. Our hypothesis assures that
$$\left(a^{[\frac13]} x_{\lambda}^* a^{[\frac13]}\right) T\left( b^{[\frac13]} y_{\mu}^* b^{[\frac13]}\right) \left(c^{[\frac13]} z_{\nu}^* c^{[\frac13]}\right)=0,$$ for every $\lambda, \mu$ and $\nu$. Taking weak$^*$ limit in $\nu$, it follows from the properties of $\pi_2^{***}$ that $$\left(a^{[\frac13]} x_{\lambda}^* a^{[\frac13]}\right) T\left( b^{[\frac13]} y_{\mu}^* b^{[\frac13]}\right) c=0,$$ for every $\lambda,$ and $\mu$. Finally, taking weak$^*$ limits first in $\mu$ and later in $\lambda$ we have $ a T^{**} (b) c=0.$ We have therefore shown that $ a T^{**}(b) c =0$ whenever $a b= b c=0$ in $M(A)$. Since $X$ is an essential $A$-bimodule, and $A^{**}$ is unital, it can be easily checked that $X^{**}$ is a unital $M(A)$-bimodule.  Thus, the statement $(a)$ follows from Theorem \ref{t LiPan} above.
\end{proof}

It is worth to notice that the proof of the above theorem makes use of the local structure of a C$^*$-algebra when it is regarded as a JB$^*$-triple. Jordan techniques are also employed in the next remark.

\begin{remark}\label{r more equivalences}{\rm In the hypothesis above, every statement in Proposition \ref{p multipliers 1} is equivalent to:
\begin{enumerate}[$(h)$]
\item For each right closed ideal $R$ of $A$ and each left closed ideal $L$ of $A$, we have $$T(R\cap L) \subseteq X L + R X.$$
\end{enumerate}

Indeed, suppose that $T^{**}: A^{**} \to X^{**}$ is a generalised derivation. Let $a$ be an element in the intersection $R\cap L.$ The element $a^{[3]} = a a^* a$ lies in $R$ and in $L$. By induction on $n$, we proved that $a^{[2n-1]}$ belongs to $L\cap R$ for every natural $n$. Therefore, the JB$^*$-subtriple, $A_a,$ of $A$ generated by $a$ is contained in $R\cap L$. Let us take $b\in A_a$ satisfying $b^{[3]} =a$. Thus, $$T(a) = T (b^{[3]}) = T (b b^* b )= T(b) b^* b + b T(b^* b) - b T^{**} (1) b^* b \in X L + R X,$$ which proves $(b)\Rightarrow (h)$.\smallskip

Assume now that $T(R\cap L) \subseteq X L + R X$ whenever $R$ is a right closed ideal of $A$ and $L$ is a left closed ideal of $A$. Pick $a,b$ and $c$ in $A$ with $a b = 0 = bc$. Let $R= \{ x\in A : a x= 0\}$ and $L=\{ x\in A : x c =0\}$. In this case, $R$ is a right closed ideal, $L$ is a left closed ideal and $b\in L\cap R$, then, by assumptions, $T(b) \in X L + R X$. Therefore, $a T(b) c \in  a X L c + a R X c = \{0\}$, which gives the desired equivalence.\smallskip
}\end{remark}

When, in the hypothesis of the above Remark \ref{r more equivalences}, $X$ coincides with $A$, we have $X L  +  R X = L + R$. Under the additional assumption of $A$ being unital, V. Shul'man established in \cite[Theorem 1]{Shu} that a bounded linear mapping $T: A\to A$ satisfies statement $(h)$ in Remark \ref{r more equivalences} if, and only if, $T = D  +L_a$, where $D: A \to A$ is a derivation and $L_a$ is the left multiplication operator by the element $a=T(1)$ (if and only if $T$ is a generalised derivation). So, our Proposition \ref{p multipliers 1} and the equivalence with Remark \ref{r more equivalences}$(h)$ provides us a generalisation of the result proved by Shul'man.\smallskip

We culminate this section exploring the connections with (local) triple derivations on a C$^*$-algebra. Let $\delta : A \to A$ be a triple derivation on a C$^*$-algebra. We have already commented that $\delta^{**} : A^{**} \to A^{**}$ is a triple derivation (see \eqref{eq bitranspose of a derivation}). By Lemma 1 in \cite{HoMarPeRu} $\delta^{**}(1)^* = -\delta^{**} (1)$. It can be easily seen that $$ \delta^{**} ( a\circ b) = \delta^{**} (\J a1b ) = \J {\delta^{**} (a)}1b + \J a{\delta^{**} (1)}b + \J a1{\delta^{**} (b)}$$ $$= {\delta^{**} (a)}\circ b + a \circ {\delta^{**} (b)} - U_{a,b} \delta^{**} (1),$$ which shows that $\delta^{**}$ is a generalised Jordan derivation and hence a generalised derivation (compare \cite[Comments before Proposition 5 and Remark 8]{BurFerGarPe2012}). Moreover, $$\delta^{**} (a^*) = \delta^{**} \J 1a1= 2 \J {\delta^{**} (1)}a1 + \J 1{\delta^{**}(a)} 1 = 2 \delta^{**} (1) \circ a^* + \delta^{**}(a)^*.$$ In particular, every triple derivation $\delta$ with $\delta^{**} (1) = 0$ is a symmetric derivation or a $^*$-derivation (i.e. $\delta$ is a derivation with $\delta (a^* ) = \delta (a)^*$, for every $a\in A$). In particular, for every triple derivation $\delta$,  $\delta^{**} -  \delta(\frac12 \delta^{**}(1),1)$ is a $^*$-derivation on $A$.

\begin{corollary}\label{c final coro triple derivations on C*-algebras} Let $T: A \to A$ be a linear operator on a C$^*$-algebra. The following are equivalent:
\begin{enumerate}[$(a)$] \item $T$ is a local triple derivation;
\item $T$ is a triple derivation;
\item $T^{**}$ is a bounded generalised derivation with $T^{**} (1) = - T^{**} (1)^*$ and $T^{**}- L_{T^{**}(1)}$ is a symmetric operator;
\item $T^{**}-L_{T^{**} (1)}$ is a bounded $^*$-derivation and $T^{**} (1) = - T^{**} (1)^*$;
\item $T^{**}$ is a bounded generalised derivation with $T^{**} (1) = - T^{**} (1)^*$ and $T^{**} -  \delta(\frac12 T^{**}(1),1)$ is a symmetric operator;
\item $T^{**} -  \delta(\frac12 \delta^{**}(1),1)$ is a bounded $^*$-derivation and $T^{**} (1) = - T^{**} (1)^*$;
\item $T^{**} (1) = - T^{**} (1)^*$, $T^{**}- L_{T^{**}(1)}$ is a symmetric operator, $T$ is bounded, and $a T(b) c = 0$, whenever $ab=bc=0$ in $A$.
\item $T^{**} (1) = - T^{**} (1)^*$, $T^{**} -  \delta(\frac12 \delta^{**}(1),1)$ is a symmetric operator, $T$ is bounded, and $a T(b) c = 0$, whenever $ab=bc=0$ in $A$.
\end{enumerate} $\hfill\Box$
\end{corollary}

\bigskip
\bigskip
\bigskip
\bigskip
\bigskip
\bigskip

\medskip

\noindent Departamento de An\'{a}lisis Matem\'{a}tico, Facultad de
Ciencias,\\ Universidad de Granada, 18071 Granada, Spain. \medskip

\noindent e-mail: aperalta@ugr.es

\end{document}